\documentclass[a4paper,12pt]{article}

\usepackage{amsfonts}
\usepackage{amscd,color}
\usepackage{amsmath,amsfonts,amssymb,amscd}
\usepackage{indentfirst,graphicx,epsfig}
\usepackage{graphicx}
\input{epsf}
\usepackage{graphicx}
\usepackage{epstopdf}
\usepackage{caption}
\usepackage{subfigure}
\usepackage{mathrsfs}
\usepackage{ifpdf}

\newtheorem{thm}{Theorem}[section]

\newtheorem{problem}[thm]{Problem}

\newtheorem{lem}[thm]{Lemma}

\newenvironment {proof} {\noindent{\em Proof.}}{\hspace*{\fill}$\Box$\par\vspace{4mm}}
\newcommand{\ml}{l\kern-0.55mm\char39\kern-0.3mm}

\title{\textbf{Erd\"{o}s-Gallai-type results for conflict-free connection of graphs\footnote{Supported by NSFC No.11871034, 11531011 and NSFQH No.2017-ZJ-790.}}}
\author{{\small Meng Ji$^1$, Xueliang Li$^{1,2}$} \\
{\small  $^1$Center for Combinatorics and LPMC}\\
{\small Nankai University, Tianjin 300071, China}\\
{\small jimengecho@163.com,}
{\small lxl@nankai.edu.cn }\\
\small $^2$School of Mathematics and Statistics, Qinghai Normal University\\
\small Xining, Qinghai 810008, China\\
}
\date{}
\begin{document}
\maketitle
\begin{abstract}
A path in an edge-colored graph is called \emph{a conflict-free path} if there exists a color used on only one of its edges. An edge-colored graph is called \emph{conflict-free connected} if there is a conflict-free path between each pair of distinct vertices. The \emph{conflict-free connection number} of a connected graph $G$, denoted by $\mathit{cfc}(G)$, is defined as the smallest number of colors that are required to make $G$ conflict-free connected. In this paper, we obtain Erd\"{o}s-Gallai-type results for the conflict-free connection numbers of graphs.
\\[2mm]
\end{abstract}

\section{Introduction}

All graphs mentioned in this paper are simple, undirected and finite. We follow book \cite{BM} for undefined notation and terminology. Let $P_1=v_1v_2\cdots v_s$ and $P_2=v_sv_{s+1}\cdots v_{s+t}$ be two paths. We denote $P=v_1v_2\cdots v_sv_{s+1}\cdots v_{s+t}$ by $P_1\odot P_2$. Coloring problems are important subjects in graph theory. The hypergraph version of conflict-free coloring was first introduced
by Even et al. in \cite{ELRS}. A hypergraph $H$ is a pair $H=(X,E)$ where $X$ is the set of vertices, and $E$ is the set of nonempty subsets of $X$, called hyper-edges. The conflict-free coloring of
hypergraphs was motivated to solve the problem of assigning frequencies to different base stations in cellular networks, which is defined as a vertex coloring of $H$ such that every hyper-edge contains
a vertex with a unique color.

Later on, Czap et al. in \cite{CJV} introduced the concept of \emph{conflict-free connection colorings} of graphs motivated by the conflict-free colorings of hypergraphs. A path in an edge-colored graph $G$ is called a \emph{conflict-free path} if there is a color appearing only once on the path. The graph $G$ is called \emph{conflict-free connected} if there is a conflict-free path between each pair of distinct vertices of $G$. The minimum number of colors required to make a connected graph $G$ conflict-free connected is called the \emph{conflict-free connection number} of $G$, denoted by $\mathit{cfc}(G)$. If one wants to see more results, the reader can refer to \cite{CDH, CJLZ, CHLMZ, CJV}. For a general connected graph $G$ of order $n$, the conflict-free connection number of $G$ has the bounds $1\leq \mathit{cfc}(G)\leq n-1$. When equality holds, $\mathit{cfc}(G)=1$ if and only if $G=K_n$ and $\mathit{cfc}(G)=n-1$ if and only if $\mathit{cfc}(G)=K_{1,n-1}$.

The Erd\"{o}s-Gallai-type problem is an interesting problem in extremal graph theory, which was studied in \cite{LLSY,LLS,LS,AL} for rainbow connection number $\mathit{rc}(G)$; in \cite{HLW} for proper connection number $\mathit{pc}(G)$; in \cite{CLW} for monochromatic connection number $\mathit{mc}(G)$. We will study the Erd\"{o}s-Gallai-type problem for the conflict-free number $\mathit{cfc}(G)$ in this paper.

\section{Auxiliary results}

At first, we need some preliminary results.

\begin{lem}\label{containing e} \upshape\cite{CJV}
Let $u,v$ be distinct vertices and let $e=xy$ be an edge of a 2-connected graph. Then there is a $u-v$ path in $G$ containing the edge $e$.
\end{lem}

For a 2-edge connected graph, the authors \cite{CHLMZ} presented the following result:
\begin{thm}\label{2-edgeconnection}\upshape\cite{CHLMZ}
If $G$ is a $2-$edge connected graph, then $\mathit{cfc}(G)=2$.
\end{thm}

For a tree $T$, there is a sharp lower bound:
\begin{thm}\label{cfc(T)>cfc(P)}\upshape\cite{CJLZ}
Let $T$ be a tree of order $n$. Then $\mathit{cfc}(T)\geq \mathit{cfc}(P_n)=\lceil\log_2n\rceil$.
\end{thm}

\begin{lem}\label{upperB}
Let $G$ be a connected graph and $H=G-B$, where $B$ denotes the set of the cut-edges of $G$. Then $\mathit{cfc}(G)\leq \max\{2,|B|\}$.
\end{lem}
\begin{proof}
If $B$=$\varnothing$, then by Theorem \ref{2-edgeconnection}, $\mathit{cfc}(G)$=2. If $|B|\geq 1$, then all the blocks are non-trivial in each component of $G-B$. Now we give $G$ a conflict-free coloring: assign one edge with color 1 and the remaining edges with color 2 in each block of each component of $G-B$; for the edges $e\in B$, we assign each edge with a distinct color from $\{1,2,\cdots,|B|\}$.

Now we check every pair of vertices. Let $u$ and $v$ be arbitrary two vertices. Consider first the case that $u$ and $v$ are in the same component of $G-B$. If $u$ and $v$ are in the same block, by Lemma \ref{containing e} there is a conflict-free $u-v$ path. If $u,v$ are in different blocks, let $P=P_1\odot P_2\odot\cdots \odot P_r$ be a $u-v$ path, where $P_i$ $(i\in[r])$ is the path in each block of the component. Then we can choose a conflict-free path in one block, say $P_1$, and choose a monochromatic path with color 2 in each block of the remaining blocks, say $P_i$ $(2\leq i\leq r-1)$, clearly, $P$ is a conflict-free $u-v$ path. Now consider the case that $u$ and $v$ are in distinct components of $G-B$. If there exists one cut-edge $e$ with color $c\notin\{1,2\}$, then there is a conflict-free $u-v$ path since the color used on $e$ is unique. If there does not exist cut-edge with color $c\notin\{1,2\}$, then suppose that there is only one cut-edge $e=xy$ with color 1, without loss of generality, let $u,x$ be in a same component and $v,y$ be in a same component. We choose a monochromatic $u-x$ path $P_1$ with color 2 and choose a monochromatic $v-y$ path $P_2$ with 2, then $P=P_1xyP_2$ is a conflict-free $u-v$ path. If there is only one cut-edge $e=st$ colored by 2, without loss of generality, then we say $u,s$ are in the same component and $t,v$ in a same component, we choose a monochromatic $u-s$ path $P_1$ and a conflict-free $t-v$ path $P_2$ in each component. Then $P=P_1stP_2$ is a conflict-free $u-v$ path. If there are exactly two cut-edges $e_1=st$ and $e_2=xy$ colored by 1 and 2, respectively, without loss of generality, we say that $u,s$ are in a same component, $t,x$ are in a same component and $y,v$ are in a same component. Then we choose a monochromatic $u,s$ path $P_1$, $t,x$ path $P_2$ and $y,v$ path $P_3$ in the three components, respectively, with color 2. Hence, $P=P_1stP_2xyP_3$ is a conflict-free $u-v$ path. So, we have $\mathit{cfc}(G)\leq \max\{2,|B|\}$.
\end{proof}

\begin{lem}\label{upperforedgeset}
Let $G$ be a connected graph of order $n$ with $k$ cut-edges. Then
\begin{center}
$|E(G)|\leq {n\choose k} +k$
\end{center}.
\end{lem}
\begin{proof}
Clearly, it holds for $k=0$. Assuming that $k\geq1$. Let $G$ be a maximal graphs with $k$ cut-edges. Let $B$ be the set of all the bridges. And let $G-B$ be the graph by deleting all the cut-edges. Let $C_1,C_2,\cdots, C_{k+1}$ be the components of $G-B$ and $n_i$ be the orders of $C_i$. Then $E(G)=\sum_{i=1}^{k+1}$$n_i\choose 2$$+k$. Let $C_i$ and $C_j$ be two components of $G-B$ with $1<n_i\leq n_j$. Now we construct a graph $G'$ by moving a vertex $v$ from $C_i$ to $C_j$, replace $v$ with an arbitrary vertex in $V(C_k)\setminus v$ for the cut-edges incident with $v$, add the edges between $v$ and the vertices in $C_j$, and delete the edges between $v$ and the vertices in $C_i$, where $v$ is not adjacent to the vertices of $C_i$. Now we have
$|E(G')|=\sum_{s=1\neq i,j}^{k+1}$$n_s\choose 2$+$n_i-1\choose 2$+$n_j+1\choose 2$+$k$=$\sum_{s=1\neq i,j}^{k+1}$$n_s\choose 2$+$n_i\choose 2$-$n_i-1$+$n_j\choose 2$+$n_j+k$=$|E(G)|+n_j-n_i+1$$>|E(G)|$. When we do repetitively the operation, we have $|E(G)|\leq {n\choose k} +k$.
\end{proof}

\section{Main results}

Now we consider the Erd\"{o}s-Gallai-type problems for $\mathit{cfc}(G)$. There are two types, see below.

\begin{problem}\label{problem1}
For each integer $k$ with $2\leq k\leq n-1$, compute and minimize the function $f(n,k)$ with the following property: for each connected graph $G$ of order $n$, if $|E(G)|\geq f(n,k)$, then $\mathit{cfc}(G)\leq k$.
\end{problem}

\begin{problem}\label{problem2}
For each integer $k$ with $2\leq k\leq n-1$, compute and maximize the function $g(n,k)$ with the following property: for each connected graph $G$ of order $n$, if $|E(G)|\leq g(n,k)$, then $\mathit{cfc}(G)\geq k$.
\end{problem}

Clearly, there are two parameters which are equivalent to $f(n,k)$ and $g(n,k)$ respectively. For each integer $k$ with $2\leq k\leq n-1$, let $s(n,k)=\max\{|E(G)|:|V(G)|=n, \mathit{cfc}\geq k\}$ and $t(n,k)=\min\{|E(G)|:|V(G)|=n, \mathit{cfc}\leq k\}$. By the definitions, we have $g(n,k)=t(n,k-1)-1$ and $f(n,k)=s(n,k+1)+1$.

Using Lemma \ref{upperB} we first solve Problem \ref{problem1}.
\begin{thm}
$f(n,k)=$$n-k-1\choose 2$$+k+2$ for $2\leq k\leq n-1$.
\end{thm}
\begin{proof} At first, we show the following claims.

\emph{Claim 1: For $k\geq2$, $f(n,k)\leq $$n-k-1\choose 2$$+k+2$.}

\emph{Proof of Claim 1:} We need to prove that for any connected graph $G$, if $E(G)\geq $$n-k-1\choose 2$$+k+2$, then $\mathit{cfc}(G)\leq k$. Suppose to the contrary that $\mathit{cfc}(G)\geq k+1$. By Lemma \ref{upperB}, we have $|B|\geq k+1$. By Lemma \ref{upperforedgeset}, $E(G)\leq$$n-k-1\choose 2$$+k+1$, which is a contradiction.

\emph{Claim 2: For $k\geq2$, $f(n,k)\geq $$n-k-1\choose 2$$+k+2$.}

\emph{Proof of Claim 2:} We construct a graph $G_k$ by identifying the center vertex of a star $S_{k+2}$ with an arbitrary vertex of $K_{n-k-1}$. Clearly, $E(G_k)=$$n-k-1\choose 2$$+k+1$. Since $\mathit{cfc}(S_{k+2})=k+1$, then $\mathit{cfc}(G_k)\geq k+1$. It is easy to see that $\mathit{cfc}(G_k)=k+1$. Hence, $f(n,k)\geq $$n-k-1\choose 2$$+k+2$.

The conclusion holds from Claims 1 and 2.
\end{proof}

Now we come to the solution for Problem \ref{problem2}, which is divided as three cases.
\begin{lem}\label{k=2}
For $k=2$, $g(n,2)$=$n\choose 2$$-1$.
\end{lem}
\begin{proof}
Let $G$ be a complete graph of order $n$. The number of edges in $G$ is $n\choose 2$, $i.e.$, $E(G)=$$n\choose 2$. Clearly, when $g(n,2)=$$n\choose 2$$-1$ for every $G$, $\mathit{cfc}(G)\geq2$.
\end{proof}

\begin{lem}\label{k<log}
For every integer $k$ with $3\leq k<\lceil\log_2n\rceil$, $g(n,k)=n-1$.
\end{lem}
\begin{proof}
We first give an upper bound of $t(n,k)$. Let $C_n$ be a cycle. Then $t(n,k)\leq n$ since $\mathit{cfc}(C_n)=2\leq k$. And then, we prove that $t(n,k)=n$. Suppose $t(n,k)\leq n-1$. Let $P_n$ be a path with size $n-1$. Since $\mathit{cfc}(P_n)=\lceil\log_2n\rceil$ by Theorem \ref{cfc(T)>cfc(P)}, it contradicts the condition the $k<\lceil\log_2n\rceil$. So $t(n,k)=n$. By the relation that $g(n,k)=t(n,k-1)-1$, we have $g(n,k)=n-1$.
\end{proof}

\begin{lem}\label{k>=log}
For $k\geq\lceil\log_2n\rceil$, $g(n,k)$ does not exist.
\end{lem}
\begin{proof}
Let $P_n$ be a path. Then we have $t(n,k)\leq n-1$ since $\mathit{cfc}(P_n)=\lceil\log_2n\rceil$. And since $t(n,k)\geq n-1$, it is clear that $t(n,k)=n-1$. Since every graph $G$ is connected, $g(n,k)\geq n-1$. By the
relation that $g(n,k)=t(n,k-1)-1$, we have $g(n,k)=n-2$ for $k\geq\lceil\log_2n\rceil$, which contradicts the connectivity of graphs.
\end{proof}

Combining Lemmas \ref{k=2}, \ref{k<log} and \ref{k>=log}, we get the solution for Problem \ref{problem2}.

\begin{thm} For $k$ with $2\leq k\leq n-1$,

$$ g(n,k)= \left\{
\begin{array}{rcl}
{n\choose 2}-1,  &      & {k=2}\\
n-1,   &      & {3\leq k<\lceil\log_2n\rceil}\\
does \ not \ exist,    &      & {\lceil\log_2n\rceil \leq k\leq n-1.}
\end{array} \right. $$

\end{thm}

\end{document}